\documentclass[preprint, aos]{imsart}
\setattribute{journal}{name}{}

\usepackage[colorlinks,citecolor=blue,urlcolor=blue]{hyperref}

\usepackage{amssymb}
\usepackage{amsmath}
\usepackage{amsthm,latexsym}
\usepackage[authoryear]{natbib}
\usepackage{graphicx}

\usepackage{times}

\usepackage{enumerate}
\newcommand{\PP}{P}

\newcommand{\lle}{\, \, {\lesssim}\, \, }

\newcommand{\eps}{\varepsilon}

\newcommand{\RR}{\mathbb{R}}

\newcommand{\NN}{\mathbb{N}}

\newcommand{\HHH}{\mathbb{H}}

\newcommand{\qv}[1]{\left< #1 \right>}

\newcommand{\BBB}{\mathbb{B}}

\newcommand{\given}{\, |\, }

\newcommand{\isum}{i_1^2+\dots+i_d^2}

\newtheorem{thm}{Theorem}[section]
\newtheorem*{condition}{Condition}
\newtheorem{lem}{Lemma}[section]

\theoremstyle{definition}

\newtheorem{ex}[thm]{Example}

\begin{document}

\begin{frontmatter}

\title{Estimating a smooth function on a large graph by Bayesian Laplacian regularisation}

\runtitle{Bayesian function estimation on graphs}

\author{Alisa Kirichenko and Harry van Zanten}
\footnote{Korteweg-de Vries Institute for Mathematics, Science Park 107, 1098 XG Amsterdam, The Netherlands. 
 Email: a.kirichenko@uva.nl, hvzanten@uva.nl. }
 \footnote{Research funded by the
Netherlands Organization for Scientific Research (NWO).}


\runauthor{Kirichenko and Van Zanten}

\maketitle

\bigskip

\begin{abstract}
We study a Bayesian approach to estimating a smooth function 
 in the context of regression or classification problems on large graphs. 
We derive theoretical results that show how asymptotically optimal Bayesian 
regularisation can be achieved under an asymptotic shape assumption on the 
underlying graph and a smoothness condition on the target function, both 
formulated in terms of the graph Laplacian. The priors we study
are randomly scaled Gaussians with precision operators involving the 
Laplacian of the graph.

\end{abstract}

\end{frontmatter}

\numberwithin{equation}{section}

\section{Introduction}
\subsection{Learning a smooth function on a large graph}

There are various problems arising in modern statistics that involve 
making inference about a ``smooth'' function on a large graph. 
The underlying graph structure in such problems can have different origins. 
Sometimes it is given by the context of the problem. This is typically the case, for instance, 
in the problem
of making inference on protein interaction networks (e.g.\ \cite{sharan}) or in image interpolation problems (\cite{liu}).
In other cases the graph is deduced from the data in a preliminary step, as
is the case with similarity graphs in label propagation methods (e.g.\ \cite{zhu2002learning}). 
Moreover, the different problems that arise in applications 
can have all kinds of different particular features. 
For example, the available data can be indexed by the vertices or by the 
edges of the graph, or both. 
Also, in some applications only partial data are available, for instance only part of 
the vertices are labeled (semi-supervised problems). 
Moreover, both regression problems and classification problems arise naturally 
in different applications. 

Despite all these different aspects, many of these problems and the methods that 
have been developed to deal with them have a number of important features in common. 
In many cases the graph is relatively ``large'' and the function of interest can be viewed
as ``smoothly varying'' over the graph. Consequently, most of the proposed
methods view the problem as a high-dimensional or nonparametric estimation problem
and employ some regularisation or penalization technique that takes the geometry 
of the graph into account and that is thought to produce an appropriate bias-variance trade-off.

In this paper we set up the mathematical framework that allows us 
to study the performance of nonparametric function estimation methods on large graphs.
We do not treat all the variants exhaustively, instead we 
consider two prototypical problems: regression, where the function of interest $f$ 
is a function on the vertices of the graph that is observed with additive noise, 
and binary classification, where a label $0$ or $1$ is observed at each vertex
and the object of interest is the ``soft label'' function $f$ whose value at a vertex $v$ 
is the probability of seeing a $1$ at $v$.
We assume the underlying graph is ``large'', in the sense that it has $n$ vertices 
for some ``large'' $n$. Our theoretical results deal with the situation that this 
number $n$ tends to infinity. Although for finite $n$ the graph has a fixed size 
and we essentially just have to estimate a Euclidean vector in $\RR^n$, it is 
useful to view the problem as high-dimensional or even nonparametric.

Despite the finite structure, it is intuitively clear that the ``smoothness'' of $f$, defined in 
a suitable manner, will have an impact on the difficulty of the problem and on 
the results that can be attained. 
Indeed, consider the extreme case of $f$ being a constant function. Then estimating
$f$ reduces to estimating a single real number. In the regression setting, for instance, 
this means that under mild conditions the sample mean gives a $\sqrt{n}$-consistent 
estimator. In the other extreme case of a completely unrestricted function there is 
no way of making any useful inference. At best we can say that in view of the 
James-Stein effect we should employ some degree of shrinking or regularisation. 
However, if no further assumptions are made, nothing can be said about consistency or rates. 
We are interested in the question what we should do in the intermediate 
situation that $f$ has some ``smoothness'' between these two extremes.

Another aspect that will have a crucial impact on the problem, in addition to the regularity of $f$, 
 is the geometry of the graph. Indeed, regular grids of different dimensions 
are special cases of the graphs we shall consider, and we know from existing theory that the best attainable 
rates for estimating a smooth function on a grid depends on the dimension 
of the grid. 
More generally, the geometry of the graph will influence the 
complexity of the spaces of ``smooth'' functions on the graph, and hence the 
performance of statistical or learning methods.

\subsection{Laplacian regularisation}

Several approaches to learning functions on graphs that have been explored in the literature 
involve regularisation using the Laplacian matrix associated with the graph (see, for example, \cite{belkin2004}, \cite{smola}, 
\cite{Hein2006}, \cite{ando}, \cite{zhu}, \cite{huang}). The graph Laplacian is 
defined as $L = D - A$, where $A$ is the adjacency matrix of the graph and 
$D$ is the diagonal matrix with the degrees of the vertices on the diagonal. 
When viewed as a linear operator, the Laplacian acts on a 
 function $f$ on the graph as
 \begin{equation}\label{eq: l}
 L f (i) = \sum_{j \sim i} \Big(f(i) - f(j)\Big), 
 \end{equation}
where we write $i \sim j$ if vertices $i$ and $j$ are connected by an edge. 
Several related operators are routinely employed as well, for
instance, the normalized Laplacian $\tilde L = D^{-1/2}L D^{-1/2}$. We will continue to 
work with $L$ in this paper, but much of the story goes through if $L$ 
is replaced by such a related operator, after minor adaptations. 

For a function $f$ on the graph the Laplacian norm is given by 
$\sum_{j\sim i} (f(i) - f(j))^2$.
Clearly, the Laplacian norm of $f$ quantifies how much the function $f$ varies
when moving along the edges of the graph. Therefore, several papers have 
proposed regularisation or penalization using this norm, as well as generalizations 
involving powers of the Laplacian or other functions, for instance, exponential ones. See, 
for example, \cite{belkin2004} or \cite{smola} and the references therein. 
There exist only few papers that study theoretical aspects of the performance of such 
methods. 
We mention, for example, \cite{belkin2004}, in which a theoretical analysis of a Tikhonov regularisation method 
is conducted in terms of algorithmic stability. 
\cite{johnson2007effectiveness} consider sub-sampling schemes for estimating 
a function on a graph.

The existing papers have different viewpoints than ours and do not study how the performance depends 
on (the combination of) graph geometry and function regularity. 
Our aim is to develop a framework which makes such a theoretical study of 
Laplacian regularisation methods possible and to 
derive some first asymptotic results that exhibit methods that perform well
from the point of view of convergence rates and adaptation to regularity.

\subsection{Bayesian regularisation}

We investigate Bayesian regularisation approaches, 
where we consider two types of priors on functions on graphs.
The first type performs regularisation using 
a power of the Laplacian. This can be seen as the graph analogue 
of Sobolev norm regularisation of functions on ``ordinary'' 
Euclidean spaces. The second type of priors 
uses an exponential function of the Laplacian. This can be 
viewed as the analogue of the popular squared exponential 
prior on functions on Euclidean space or its extension
to manifolds, as studied by \cite{castillo2014thomas}. 
In both cases we consider hierarchical priors with 
the aim of achieving automatic adaptation to the regularity 
of the function of interest.

To assess the performance of our Bayes procedures we take an asymptotic
perspective. We let the number of vertices of the graph grow and ask 
at what rate the posterior distribution concentrates around the unknown 
function $f$ that generates the data. We make two kinds of assumptions. 
Firstly, we assume that $f$ has a certain degree of regularity $\beta$, 
defined in suitable manner. The smoothness $\beta$ is not 
assumed to be known though, we are aiming at deriving adaptive results. 

Secondly, we make an assumption 
on the asymptotic shape of the graph. In recent years, various theories of graph limits have been 
developed. Most prominent is the concept of the graphon, e.g. \cite{lovasz_limits} or 
 the book of \cite{lovasz}. More recently this notion has been extended in various directions, 
 see, for instance, \cite{borgs} and \cite{chung}.
 However, the existing approaches are not immediately suited in the situations we have in mind, 
which involve graphs that are sparse in nature and are ``grid-like'' in some sense. 
Therefore we take an alternative approach and describe the asymptotic shape 
of the graph through a condition on the asymptotic behaviour of the 
spectrum of the Laplacian. To be able to derive concrete results 
we essentially assume that the smallest eigenvalues $\lambda_{n,i}$ of $L$ satisfy 
\begin{equation}\label{eq: c11}
\lambda_{n,i}^2 \asymp \Big(\frac{i}{n}\Big)^{2/r}
\end{equation}
for some $r \ge 1$\footnote{We write $a_n \asymp b_n$ if $0 < \liminf a_n/b_n \le \limsup a_n/b_n < \infty$.}. 
Very roughly speaking, this 
means that asymptotically, or ``from a distance'', the graph looks like 
an $r$-dimensional grid with $n$ vertices. 
As we shall see, the actual grids are special cases (see Example \ref{ex: grids}), 
hence our results include the usual 
statements for regression and classification on these classical design spaces. 
However, the setting is much more general, since it is really only the {\em asymptotic} shape
that matters. For instance, a $2$ by $n/2$ ladder graph asymptotically also looks 
like a path graph, and indeed we will see that it satisfies our assumption for $r=1$ as well
(Example \ref{ex: ladder}). 
Moreover, the constant $r$ in \eqref{eq: c11} does not need to be 
a natural number. We will see, for example, at least numerically, that there are graphs whose geometry 
is asymptotically like that of a grid of non-integer ``dimension'' $r$ 
in the sense of condition \eqref{eq: c11}. 

We stress that we do not assume the existence of a ``limiting manifold'' for the graph as $n \to \infty$. 
We formulate our conditions and results purely in terms of intrinsic properties of the graph, 
without first embedding it in an ambient space. 
In certain cases in which limiting manifolds do exist (e.g.\ the regular grid cases) our 
type of asymptotics can be seen as ``infill asymptotics'' (\cite{cressie1993statistics}). 
For a simple illustration, see Example \ref{ex: pathgraph}. 
 However, in applied settings (see, for instance, Example \ref{ex: protein}) it is typically not
clear what a suitable ambient manifold could be, which is why we choose to avoid 
this issue altogether.

In the recent paper \cite{Jarno} the theoretical results we present in this paper
are investigated numerically and serve as a guideline for the tuning of practical Bayesian 
regularisation methods. Several concrete examples  are considered, 
both for simulated data and for real data problems.

\subsection{Organisation}
The rest of the paper is organized as follows. 
In the next section we present our geometry assumption and give 
examples of graphs that satisfy it, either theoretically or numerically.
In Section \ref{general_result} we introduce two families of priors on functions 
on graphs. We present theorems that quantify the amount of mass that the priors
put on neighbourhoods of ``smooth'' functions and quantify
the complexity of the priors in terms of metric entropy.
Section \ref{sec: proofs} contains the proofs of these general results 
and in Section \ref{sec: stat} they are used to derive theorems 
about posterior contraction in nonparametric regression and binary classification.
We end with some concluding remarks in Section \ref{sec: conc}.

\section{Asymptotic geometry assumption on graphs}
\label{deff}

In this section we formulate our geometry assumption on the underlying graph 
and give several examples.

\subsection{Graphs, Laplacians and functions on graphs}

Let $G$ be a connected, simple (i.e.\ no loops, multiple edges or weights), undirected graph with $n$ vertices labelled $1, \ldots, n$. 
Let $A$ be its adjacency matrix, i.\,e. $A_{ij}$ is $1$ or $0$ according 
to whether or not there is an edge between vertices $i$ and $j$. Let $D$ be the diagonal matrix 
with element $D_{ii}$ equal to the degree of vertex $i$. Let $L = D-A$ be the Laplacian of 
the graph. 
We note that strictly speaking, we will be considering sequences of graphs $G_n$
with Laplacians $L_n$ and we will let $n$ tend to infinity. However, in order to avoid cluttered notation, we will omit the subscript $n$ and just write $G$ and $L$ throughout.

A function $f$ on the (vertices of the) graph is simply a function $f: \{1, \ldots, n\} \to \RR$. 
Slightly abusing notation we will write $f$ both for the function and for 
the associated vector of function values $(f(1), f(2), \ldots, f(n))$ in $\RR^n$. 
We measure distances and norms of functions using the norm $\|\cdot\|_n$ defined by 
$\|f\|^2_n = n^{-1} \sum_{i=1}^n f^2(i)$.
The corresponding inner product of two functions $f$ and $g$ is denoted by 
\[
\qv{f,g}_n = \frac1n \sum_{i=1}^n f(i)g(i).
\]
Again, in our results $n$ will be varying, so when we speak of a function $f$ 
on the graph $G$ we are, strictly speaking, considering a sequence of functions $f_{n}$. 
Also, in this case the subscript $n$ will usually be omitted.

The Laplacian $L$ is positive semi-definite and symmetric. 
It easily follows from the definition that its smallest eigenvalue is $0$ (with 
eigenvector $(1, \ldots, 1)$). The fact that $G$ is connected implies 
that the second smallest eigenvalue, the so-called algebraic connectivity, is strictly 
positive (e.g.\ \cite{thebook}). We will denote the Laplacian eigenvalues, 
ordered my magnitude, by 
\[
0 = \lambda_{n, 0} < \lambda_{n,1} \le \lambda_{n,2} \le \cdots \le \lambda_{n,n-1}.
\]
Again we will usually drop the first index $n$ and just write $\lambda_{i}$ for $\lambda_{n, i}$. 
We fix a corresponding sequence of eigenfunctions $\psi_{i}$, orthonormal with respect to 
the inner product $\qv{\cdot, \cdot}_n$.

\subsection{Asymptotic geometry assumption}
\label{sec: geometry}

As mentioned in the introduction, we will derive results under an asymptotic shape 
assumption on the graph, formulated in terms of the Laplacian eigenvalues.
To motivate the definition we note that the $i$th eigenvalue of the Laplacian of an $n$-point grid 
of dimension $d$ behaves like $(i/n)^{2/d}$ (see Example \ref{ex: grids} ahead). 
We will work with the following condition.

\begin{condition}
We say that the {\em geometry condition is satisfied with parameter $r \ge 1$} 
if there exist $i_0 \in \NN$, $\kappa \in (0,1]$ and $C_1, C_2 > 0$ such that 
for all $n$ large enough, 
\[
C_1\Big(\frac i n\Big)^{2/r} \le \lambda_i \le C_2\Big(\frac i n\Big)^{2/r}, 
 \qquad \text{for all $i \in \{i_0, \ldots, \kappa n\}$}.
\]
\end{condition}

Note that this condition only restricts a positive fraction $\kappa$
of the Laplacian eigenvalues, namely the $\kappa n$ smallest ones. Moreover, 
we don't need restrictions on the first finitely many eigenvalues. 
We remark that if the geometry condition is fulfilled, then by adapting the constant
$C_1$ we can ensure that the lower bound holds, in fact, for {\em all} $i \in \{i_0, \ldots, n\}$.
To see this, observe that for $n$ large enough and $\kappa n < i \le n$ we have 
\[
\lambda_i \ge \lambda_{\lfloor \kappa n\rfloor} \ge C_1\Big(\frac {\lfloor \kappa n\rfloor} n\Big)^{2/r}
\ge C_1\Big(\frac\kappa 2\Big)^{2/r}\Big(\frac i n\Big)^{2/r}.
\]
For the indices $i < i_0$ it is useful to note that we have a general lower bound 
on the first positive eigenvalue $\lambda_1$, hence on $\lambda_2, \ldots, \lambda_{i_0}$ as well. 
Indeed, by Theorem 4.2 of \cite{mohar2} we have 
\begin{equation}\label{eq: l1bound}
\lambda_1 \ge \frac{4}{n\, \text{diam}(G)} \ge \frac{4}{n^2}.
\end{equation}
Note that this bound also implies that our geometry assumption can not 
hold with a parameter $r < 1$, since that would lead to 
contradictory inequalities for $\lambda_{i_0}$.

We first confirm that 
the geometry condition is satisfied 
 for grids and tori of different dimensions.

\begin{ex}[Grids]
\label{ex: grids}
For $d \in \NN$, a regular $d$-dimensional grid with $n$ vertices 
can be obtained by taking the Cartesian product of $d$ path graphs 
with $n^{1/d}$ vertices (provided, of course, that this number is an integer). 
Using the known expression for the Laplacian eigenvalues of the path graph 
and the fact that the eigenvalues of products of graphs are the sums of the original 
eigenvalues, see, for instance, Theorem 3.5 of \cite{mohar}, we get that the Laplacian eigenvalues of the $d$-dimensional grid are 
given by 
\[
4\left(\sin^2{\frac{\pi i_1}{2n^{\frac 1 d}}}+ \dots + \sin^2{\frac{\pi i_d}{2n^{\frac 1 d}}}\right) \asymp \frac{\isum}{n^{2/d}}, 
\]
where $i_k= 0, 1,2, \ldots, n^{1/d}-1$ for every $k =1, \ldots, d$. 
By definition there are $i+1$ eigenvalues less or equal than the $i$th smallest eigenvalue 
$\lambda_i$. Hence, for a constant $c > 0$, we have: 
\[
i+1 = \sum_{i_1^2+ \cdots + i_d^2 \le c^2 n^{2/d}\lambda_i} 1.
\]
 The sum on the right gives the number of lattice points in a sphere of radius 
$R= c n^{1/d}\sqrt\lambda_i$ in $\RR^d$. For our purposes it suffices to use crude upper 
and lower bounds for this number.
By considering, for instance, the smallest hypercube
containing the sphere and the largest one inscribed in it, it is easily seen that the 
number of lattice points 
is bounded from above and below by a constant times $R^d$. We conclude that for the $d$-dimensional 
grid we have $\lambda_i \asymp (i/n)^{2/d}$ for every $i =0, \ldots, n-1$. 
In particular, the geometry condition is fulfilled with parameter $r = d$. 
\end{ex}

\begin{ex}[Discrete tori]

For graph tori we can follow the same line of reasoning as for grids.
A $d$-dimensional torus graph with $n$ vertices can be obtained as a 
product of $d$ ring graphs with $n^{1/d}$ vertices. Using the known explicit 
expression of the Laplacian eigenvalues of the ring we find that the 
$d$-dimensional torus graph satisfies the geometry conditions with parameter $r = d$ as well.
\end{ex}

The following lemma lists a number of operations
that can be carried out on a graph without 
loosing the geometry condition.

\begin{lem}
\label{plusminus}
Suppose that $G=G_n$ satisfies the geometry assumption with parameter $r$. Then 
the following graphs satisfy the condition with parameter $r$ as well:
\begin{enumerate}[(i)]
\item 
The cartesian product of $G$ with a connected simple graph $H$ with a finite number of vertices
(independent of $n$).
\item 
The graph obtained by augmenting $G$ with finitely many edges (independent of $n$), 
provided it is a simple graph.
\item 
The graph obtained from $G$ by deleting finitely many edges (independent of $n$), provided it is 
still connected.
\item 
The graph obtained by augmenting $G$ with finitely many vertices and edges (independent of $n$), 
provided it is a simple connected graph.
\end{enumerate}
\end{lem}

\begin{proof}
(i). Say $H$ has $m$ vertices and let its
 Laplacian eigenvalues be denoted by $0=\mu_0, \ldots, \mu_m$. Then the
product graph has $mn$ vertices and it has Laplacian eigenvalues 
$\lambda_i+\mu_j$, $i=0, \dots, n-1, j=0, \dots, m-1$
(see Theorem 3.5 of \cite{mohar}). In particular, 
the first $n$ eigenvalues are the same as those of $G$. Hence, since $G$ satisfies the geometry condition, 
so does the product of $G$ and $H$. 

(ii) and (iii). 
These statements follow from the interlacing formula that asserts that 
if $G+e$ is the graph obtained by adding the edge $e$ to $G$, then 
\[
0\leq\lambda_1(G)\leq\lambda_1(G+e)\leq\lambda_2(G)\leq\dots\leq\lambda_{n-1}(G)\leq\lambda_{n-1}(G+e). 
\]
See, for example, Theorem 3.2 of \cite{mohar} or Theorem 7.1.5 of \cite{thebook}.

(iv). Let $v$ and $e$ be a vertex and an edge that we want to connect to $G$. 
Denote $G_v$ a disjoint union of $G$ and $v$, and by $G'$ the graph obtained by 
connecting edge $e$ to $v$ and an existing vertex of $G$. 
By Theorem 3.1 from \cite{mohar} we know that the eigenvalues of $G_v$ are $0, 0, \lambda_1(G), \lambda_2(G), \dots, \lambda_{n-1}(G). $ Using Theorem 3.2 of \cite{mohar} we see that $0=\lambda_0(G_v)=\lambda_0 (G')$ and
\[
0=\lambda_1(G_v)\leq \lambda_1(G')\leq \lambda_1(G)\leq\lambda_2(G_v)\leq\dots\leq\lambda_{n-1}(G)\leq\lambda_n(G'). 
\]
The result follows from this observation.
\end{proof}

\begin{ex}[Ladder graph]
\label{ex: ladder}

A ladder graph with $n$ vertices is the product of a path graph with $n/2$
vertices and a path graph with $2$ vertices. Hence, by part (i) of Lemma 
\ref{plusminus} and Example \ref{ex: grids} it satisfies the geometry 
condition with parameter $r=1$. 
\end{ex}

\begin{ex}[Lollipop graph]
The so-called lollipop graph $L_{m,n}$ is obtained by attaching a path graph
with $n$ vertices with an additional edge to a complete graph with $m$ vertices. 
If $m$ is constant, i.e.\ independent of $n$, then according to parts (ii) and 
(iv) of the preceding lemma this graph satisfies the geometry condition with $r=1$. 
\end{ex}

In the examples considered so far it is possible to verify theoretically 
that the geometry condition is fulfilled. In a concrete case in which 
the given graph is not of such a tractable type, numerical investigation 
of the Laplacian eigenvalues can give an indication as to whether or 
not the condition is reasonable and provide the appropriate value of the parameter $r$.
A possible approach is to plot $\log \lambda_i$ against $\log (i/n)$. If the
geometry condition is satisfied with parameter $r$, the $\kappa \times 100 \%$ 
left mosts points in this plot should approximately
lie on a straight line with slope $2/r$, except possibly a few on the very left.

Our focus in this paper is not on numerics, but it is illustrative 
to consider a few numerical examples in order to get a better idea of the types
of graphs that fit into our framework.

\begin{ex}[Two-dimensional grid, numerically]
Figure \ref{fig: grid} illustrates the suggested numerical approach for a two-dimensional, 
$20\times20$ grid. The dashed line in the left panel is fitted to the left-most $35\%$ of the points
in the plot, discarding the first three points on the left. 
In accordance with Example \ref{ex: grids} this line has slope $1.0$. 
\begin{figure}
\begin{center}
\includegraphics[scale=.6]{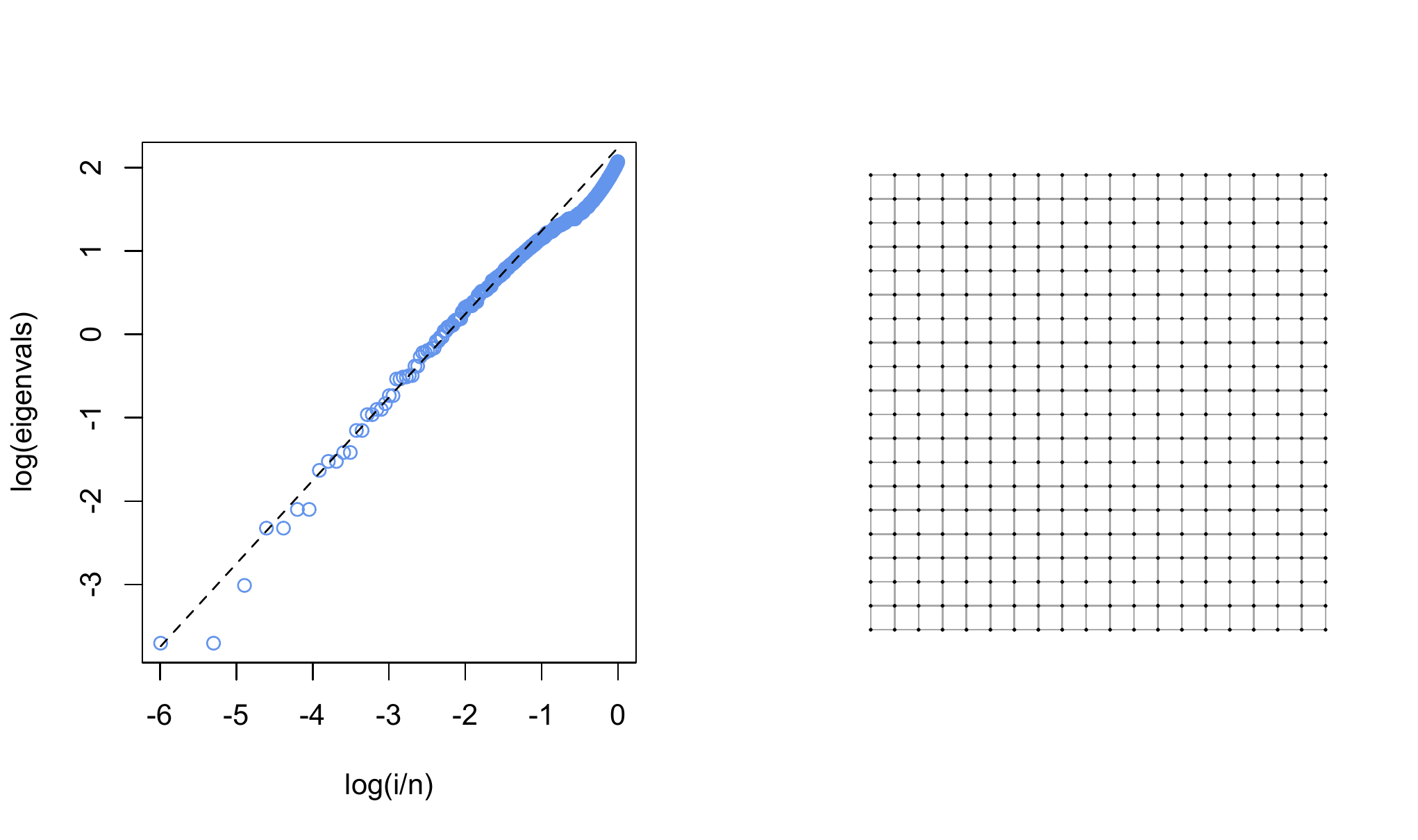}
\end{center}
\caption{Plot of $\log\lambda_i$ against $\log(i/n)$ for the $20\times20$ grid. Fitted
line has slope $1.0$, corresponding to $r=2.0$ in the geometry assumption.}
\label{fig: grid}
\end{figure}
\end{ex}

\begin{ex}[Watts-Strogatz `small world' graph]
In our second numerical example we consider a graph obtained as a realization 
from the well-known random graph model of \cite{watts1998collective}. Specifically, 
we consider in the first step a ring graph with $200$ vertices. In the 
next step every vertex is visited and the edges emanating from the vertex 
are rewired with probability $p=1/4$, meaning that with probability $1/4$ they are detached from the 
neighbour of the current vertex and attached to another vertex, chosen uniformly
at random. In the right panel of Figure \ref{fig: ws} a particular realization is
shown. Here we have only kept the largest connected component, which has $175$ vertices 
in this case. On the left we have exactly the same plot as described in the preceding example
for the grid case. The plot indicates that it is not unreasonable to assume that the geometry 
condition holds. The value of the parameter $r$ deduced from the slope
of the line equals $1.4$ for this graph.
\begin{figure}
\begin{center}
\includegraphics[scale=.6]{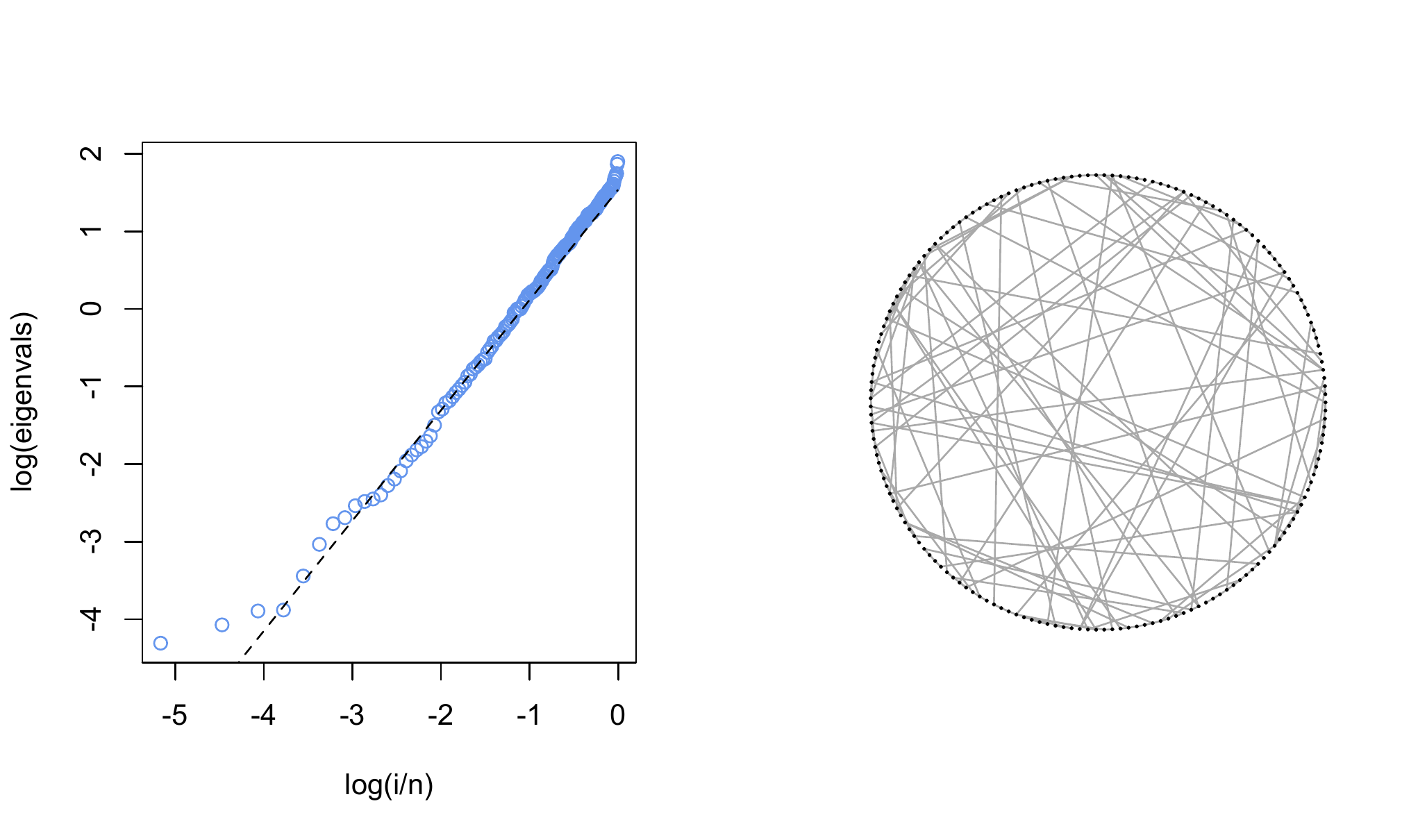}
\end{center}
\caption{Plot of $\log\lambda_i$ against $\log(i/n)$ for the Watts-Strogatz graph in the 
right panel. Fitted line has slope $1.42$, corresponding to $r=1.4$ in the geometry assumption.}
\label{fig: ws}
\end{figure}
\end{ex}

\begin{ex}[Protein interaction graph]\label{ex: protein}
In the final example we consider a graph obtained from the 
protein interaction graph of baker's yeast, as described
in detail in Section 8.5 of \cite{Kolaczyk}.
The graph, shown in the right panel 
of Figure \ref{fig: protein}, describes the interactions between proteins 
 involved in the communication between a cell and its surroundings.
Also for this graph it is true that with a few exceptions, the points corresponding
to the $35\%$ smallest eigenvalues lie approximately on a straight line. 
The same procedure as followed in the other examples gives a value $r = 2.1$
for the parameter in the geometry assumption.

\begin{figure}
\begin{center}
\includegraphics[scale=.6]{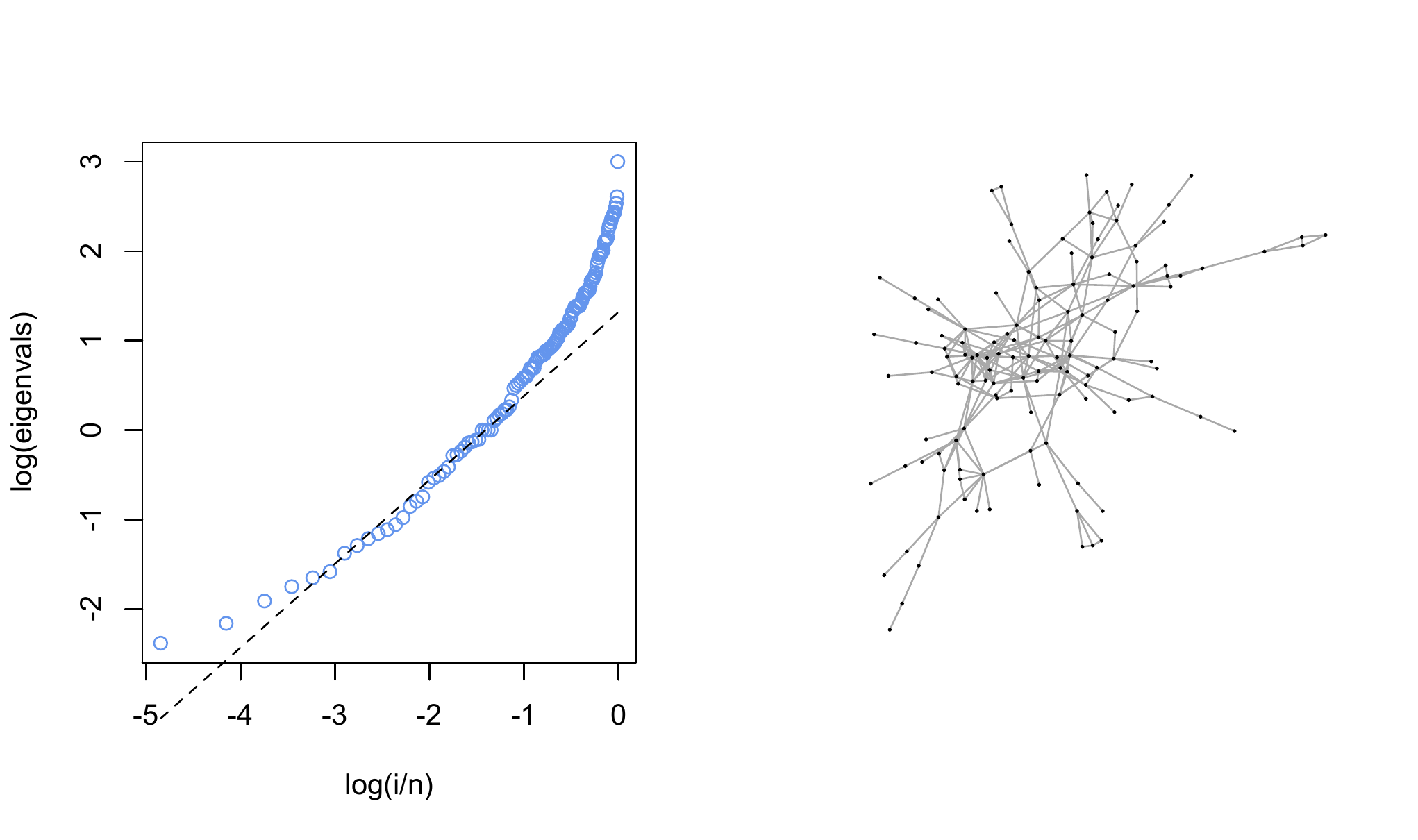}
\end{center}
\caption{Plot of $\log\lambda_i$ against $\log(i/n)$ for the Protein interaction graph in the 
right panel. Fitted line has slope $0.94$, corresponding to $r=2.1$ in the geometry assumption.}
\label{fig: protein}
\end{figure}

\end{ex}



\section{General results on prior concentration}
\label{general_result}

We consider two different priors on functions on graphs.
The first corresponds to regularisation using a power of the Laplacian, 
the second one uses an exponential function of the Laplacian. 
In this section we present two general results which quantify 
both the mass that these priors give to shrinking $\|\cdot\|_n$-neighbourhoods of a fixed 
function $f_0$, and the complexity of the support of the priors, 
measured in terms of metric entropy\footnote
{For $\eps>0$ and a norm $\|\cdot\|$ on a set $B$, we denote by $N(\eps, B, \|\cdot\|)$ 
the minimal number of balls of $\|\cdot\|$-radius $\eps$ needed to cover $B$. 
}.
In the next section we will combine these results with known results 
from Bayesian nonparametrics theory to deduce convergence rates and adaptation 
for nonparametric regression and classification problems on graphs.

Our results assume that the geometry condition holds for some $r \ge 1$. 
The mass a prior puts near $f_0$ will depend on the ``regularity'' of the function, defined
in a suitable manner. Specifically, we will assume it belongs to a Sobolev-type ball 
of the form 
\begin{equation}\label{eq: h}
H^\beta(C) = \Big\{f: \qv{f, (I+(n^\frac{2}{r} L)^\beta) f}_n \le C^2\Big\}
\end{equation}
for some $\beta, C > 0$ (independent of $n$). 
The particular normalisation, which depends on the 
geometry parameter $r$, ensures non-trivial asymptotics. 
This is confirmed in \cite{minimax}, in which  minimax lower
bounds are presented which complement the rate results of the present paper.

It is illustrative to consider the assumption in a bit more detail 
in the simple case of the path graph. The example shows in particular that 
we have chosen the ``correct'' normalisation in the definition of the smoothness class.

\begin{ex}[Path graph]
\label{ex: pathgraph}
Consider a path graph $G$ with $n$ vertices, which we identify 
with the points $i/n$ in the unit interval, $i = 1, \ldots, n$.
As seen in Example \ref{ex: grids}, this graph satisfies the 
geometry condition with parameter $r = 1$. 
Hence, in this case the collection of functions $H^\beta(C)$ is given by 
\[
H^\beta(C) = \Big\{f: \qv{f, (I+(n^2 L)^\beta) f}_n \le C^2\Big\}.
\]
To understand when a (sequence of) function(s) belongs to this space,
say for $\beta =1$, let $f_n$ be the restriction to the grid $\{i/n, i =1, \ldots n\}$
of a fixed function $f$ defined on the whole interval $[0,1]$. 
The assumption that $f_n \in H^1(C)$ then translates to the requirement that
\[
\frac1n\sum_i f^2(i/n) + n\sum_{i \sim j} (f(i/n)-f(j/n))^2 \le C^2.
\]
The first term on the left is a Riemann sum which approximates the integral 
$\int_0^1 f^2(x)\,dx$.
If $f$ is differentiable, then for the second term we have, for large $n$, 
\[
n\sum_{i \sim j} (f(i/n)-f(j/n))^2 = n\sum_{i=1}^{n-1} (f((i+1)/n)-f(i/n))^2
\approx \frac1n\sum_i (f'(i/n))^2,
\]
which is a Riemann sum that approximates the integral $\int_0^1 (f'(x))^2\,dx$.
Hence in this particular case the space of functions $H^1(C)$ on the graph is the 
natural discrete version of the usual Sobolev ball
\[
\Big\{f:[0,1] \to \RR: \int_0^1 (f^2(x)+f'^2(x))(x)\,dx \le C^2 \Big\}.
\]
Definition \eqref{eq: h} is a way of describing ``$\beta$-regular'' functions on a 
general graph satisfying the geometry condition, 
without assuming the graph or the function on it are discretised 
versions of some ``continuous limit''.
\end{ex}

The first family of priors we consider penalize the higher order Laplacian 
norm of the function of interest. This corresponds 
to using a Gaussian prior with a power of the Laplacian as precision matrix
(inverse covariance). 
(We note that since the Laplacian always has $0$ as an eigenvalue, 
it is not invertible. We remedy this by adding a small multiple of the identity matrix $I$ 
to $L$.) The larger the power of the Laplacian used, 
 the more ``rough'' functions on the graph are penalized. 
The power is regulated by a hyperparameter $\alpha > 0$ which can be seen as 
describing the ``baseline regularity'' of the prior. 
To enlarge the range of regularities for which we obtain good contraction
rates in the statistical results, we add a multiplicative hyperparameter which 
we endow with a suitable hyperprior. In \eqref{eq: c1} we assume an exact standard 
exponential distribution, but inspection of the proof shows that the range 
of priors for which the result holds is actually larger. To keep 
the exposition clean we omit these details.

\begin{thm}[Power of the Laplacian]
\label{main_theorem}
Suppose the geometry assumption holds for $r \ge 1$. 
Let $\alpha >0$ be fixed and assume that $f_0\in H^\beta(C)$ for some $C>0$ and $0<\beta\leq \alpha+r/2$. 
Let the random function $f$ on the graph be defined by 
\begin{align}
\label{eq: c1} c & \sim \text{Exp}(1)\\
\label{eq: f1} f \given c & \sim N(0, ( ((n/c)^{{2}/r}(L + n^{-2} I))^{\alpha+r/2})^{-1}).
\end{align}
Then there exists a constant $K_1>0$ and for all $K_2>1$ there exist Borel measurable subsets $B_n$ of 
$\RR^n$ such that for every sufficiently large $n$, 
\begin{align}
\PP(\|f-f_0\|_n<\eps_n) & \geq e^{-K_1 n\eps_n^2}, \label{prior_mass1} \\
\PP(f \notin B_n) & \le e^{-K_2 n\eps^2_n}, \label{remaining_mass1} \\
\log N(\eps_n, B_n, \|\cdot\|_n) & \leq n\eps^2_n \label{entropy1},
\end{align}
where $\eps_n$ is a multiple of $n^{-{\beta}/({2\beta+r})}$.
\end{thm}

Note that in this theorem we obtain the rate $n^{-{\beta}/({2\beta+r})}$
for all $\beta$ in the range $(0, \alpha +r/2]$. In the statistical 
results presented in Section \ref{sec: stat} this translates into 
rate-adaptivity up to regularity level $\alpha + r/2$. So by putting a 
prior on the multiplicative scale we achieve a degree of adaptation, 
but only up to an upper bound that is limited by our choice of the 
hyperparameter $\alpha$. 
A possible solution is to consider other functions of the Laplacian 
instead of using a power of $L$ in the prior specification. Here we 
consider usage of an exponential function of the Laplacian. We 
include a ``lengthscale'' or ``bandwidth'' hyperparameter that we 
endow with a prior as well for added flexibility.
This prior can be seen as the analogue of the prior used in \cite{castillo2014thomas} 
in the context of function estimation on manifolds, which in turn 
is a generalization of the popular squared exponential Gaussian prior 
used for estimation functions on Euclidean domains (e.g.\ \cite{rasmussen}). 
However, we stress again 
that we do not rely on an embedding of our graph in a manifold or the existence of 
a ``limiting manifold''.

In the next theorem there is indeed no restriction on the range of 
the smoothness $\beta$. We remark however that we obtain an additional 
logarithmic factor is the rate. Technically this is a consequence of the 
larger ``complexity'' of the support of this prior.

\begin{thm}[Exponential of the Laplacian]
\label{main_theorem2}
Suppose the geometry assumption holds for $r \ge 1$. 
Assume that $f_0\in H^\beta(C)$ for some $C>0$ and $\beta > 0$. 
Let the random function $f$ on the graph be defined by 
\begin{align}
\label{eq: c2} c & \sim \text{Exp}(1)\\
\label{eq: f2} f \given c &\sim N(0, ne^{-(n/c)^{2/r}L}).
\end{align}
Then there exists a constant $K_1>0$ and for all $K_2>1$ there exist Borel subsets $B_n$ of 
$\RR^n$ such that for every sufficiently large $n$, 
\begin{align}
\PP(\|f-f_0\|_n<\eps_n) & \geq e^{-K_1 n\eps_n^2}, \label{prior_mass} \\
\PP(f \notin B_n) & \le e^{-K_2 n\eps^2_n}, \label{remaining_mass} \\
\log N(\tilde\eps_n, B_n, \|\cdot\|_n) & \le n\tilde\eps^2_n \label{entropy},
\end{align}
where $\eps_n = (n/\log^{1+r/2}n)^{-\beta/(2\beta+r)}$ and $\tilde \eps_n = \eps_n \log^{1/2+r/4}n$.
\end{thm}

\section{Proofs of Theorems \ref{main_theorem} and \ref{main_theorem2}}
\label{sec: proofs}

Recall that we identify functions on the graph with vectors in $\RR^n$. 
In both cases we have that given $c$, the random vector $f$ is a centered 
$n$-dimensional Gaussian random vector. 
We view this as a Gaussian random element in the space $(\RR^n, \|\cdot\|_n)$. 
The corresponding RKHS $\HHH^c$ is the entire space $\RR^n$, and the corresponding
RKHS-norm is given by 
\[
\|h\|^2_{\HHH^c} = h^T\Sigma_c^{-1} h, 
\]
where $\Sigma_c$ is the covariance matrix of $f \given c$. (See e.g.\ \cite{RKHS}
for the definition and properties of the RKHS.)
Recall that the $\psi_i$ are the eigenfunctions of $L$, normalised with respect 
to the norm $\|\cdot\|_n$. They are then also eigenfunctions of $\Sigma_c^{-1}$
in both cases. We 
denote the corresponding eigenvalues by $\mu_i$.

The Gaussian $N(0, \Sigma_c)$ admits the series representation 
\begin{equation}\label{eq: series}
\sum Z_i\psi_i/\sqrt{n\mu_i}, 
\end{equation}
where $Z_1, \ldots, Z_n$ are standard normal variables.
In particular the functions $\psi_i/\sqrt{n\mu_i}$ form an orthonormal basis of the 
RKHS $\HHH^c$. Hence, the ordinary $\|\cdot\|_n$-norm and the RKHS-norm of a function $h$ with expansion
$h = \sum h_i \psi_i$ are given by 
\begin{equation}\label{eq: norms}
\|h\|_n^2 = \sum\limits_{i=0}^{n-1} h_i^2, \qquad 
\|h\|^2_{\HHH^c} = n \sum_{i=0}^{n-1}\mu_ih_{i}^2. 
\end{equation}
We denote the unit ball of the RKHS by $\HHH^c_1 = \{h \in \HHH^c: \|h\|_{\HHH^c} \le 1\}$.

\subsection{Proof of Theorem \ref{main_theorem}}

In this case $\Sigma_c^{-1} = ((n/c)^{{2}/r}(L+ n^{-2} I))^{\alpha+r/2}$ is the precision matrix of $f$ given $c$ 
and the eigenvalues of $\Sigma_c^{-1}$ are given by 
\[
\mu_i = \Big(\Big(\frac nc\Big)^{{2}/r}\Big(\lambda_i+\frac1{n^2}\Big)\Big)^{\alpha+r/2}.
\]

\subsubsection{Proof of \eqref{prior_mass1}}

By Lemma 5.3 of \cite{RKHS}, it follows from Lemmas \ref{lemma_small1} and \ref{lemma_balls1} ahead that 
under the conditions of the theorem and for $\eps= \eps_n=n^{-\beta/(r+2\beta)}$ 
and $c = c_n$ satisfying 
$\sqrt n \eps_n^{(\beta-\alpha)/\beta} \le c^{(\alpha+ r/2)/r}_n \le 2 \sqrt n \eps_n^{(\beta-\alpha)/\beta}$, 
we have
\[
-\log \PP(\|f - f_0\|_n \given c) \lle \eps_n^{-r/\beta}.
\]
By conditioning, it is then seen that
\begin{align*}
\PP(\|f-f_0\|_n < \eps_n) \ge e^{- K_0\eps_n^{-r/\beta}}
\int_{(\sqrt n\eps_n^{({\beta-\alpha})/\beta})^{r/(\alpha+r/2)}}^{(2\sqrt n\eps_n^{({\beta-\alpha})/ \beta})^{r/(\alpha +r/2)}} e^{-x}\,dx 
\ge e^{- K_1\eps_n^{-r/\beta}},
\end{align*}
for constants $K_0, K_1 > 0$.

\begin{lem}
\label{lemma_small1}
For $n$  large enough and $\eps > 0$ and $\eps\sqrt n/ c^{(\alpha + r/2)/r}$ small enough,
\begin{equation}
\label{small_ball}
-\log \PP(\|f\|_n \le \eps \given c) \lle \Big(\frac{c^{(\alpha + r/2)/r}}{\eps\sqrt n}\Big)^{\frac r\alpha}. 
\end{equation}
\end{lem}

\begin{proof}
By the series representation \eqref{eq: series} we have 
$\PP(\|f\|_n \le \eps \given c) = \PP(\sum Z_i^2/(n\mu_i) \le \eps^2)$. 
Recall from Section \ref{sec: geometry} that we can assume without loss of generality 
that we have the lower bounds
\begin{align}
\label{eq: lb1} \lambda_i & \ge C_1 \Big(\frac1n\Big)^2, \qquad \ \ \ 1 \le i \le i_0,\\
\label{eq: lb2} \lambda_i & \ge C_1 \Big(\frac in\Big)^{2/r}, \qquad i > i_0.
\end{align}
These translate into lower bounds for the $\mu_i$ from which it follows that 
for $\eps >0$, 
\begin{align*}
\PP(& \|f\|^2_n \le 2\eps^2 \given c) \ge 
\PP\Big(\sum_{i \le i_0} \frac{Z_i^2}{n\mu_i} \le \eps^2, \sum_{i > i_0}\frac{Z_i^2}{n\mu_i} \le \eps^2\Big)\\
& \ge \PP\Big(\sum_{1<i \le i_0} {Z_i^2} \le (C_1^pc^{-2p/r}{n^{(2\alpha+2r-2pr)/r}})\eps^2\Big)
\PP\Big(\sum_{i > i_0}\frac{Z_i^2}{i^{2p/r}} \le C_1^pc^{-2p/r}n\eps^2\Big),
\end{align*}
where we write $p = \alpha + r/2$.
By Corollary 4.3 from \cite{dunker}, the last factor in the last line is bounded
form below by 
\[
\exp\Big(-\text{const}\times (c^{-p/r}\eps\sqrt{n})^{-r/\alpha}\Big),
\]
provided $c^{-p/r}\eps\sqrt{n}$ is small enough.
By the triangle inequality and independence, the first factor is bounded from below 
by 
\[
\Big(\PP(|Z_1| \le i^{1/2}_0C_1^{p/2}c^{-p/r}{n^{(\alpha+r-pr)/r}}\eps)\Big)^{i_0}.
\]
Since $r \ge 1$, we have $c^{-p/r}{n^{(\alpha+r-pr)/r}}\eps = O(c^{-p/r}\eps\sqrt n)$. Hence,
for $c^{-p/r}\eps\sqrt n$ small enough the probability is further bounded from 
below by 
\[
\text{const} \times \Big(c^{-p/r}{n^{(\alpha+r-pr)/r}}\eps\Big)^{i_0}.
\]
Combining the bounds for the separate factors we find that, for $c^{-p/r}\eps\sqrt n$ small enough,
\[
-\log \PP(\|f\|^2_n \le 2\eps^2 \given c) \lle \log\Big(\frac{c^{p/r}}{n^{(\alpha+r-pr)/r}\eps}\Big) 
+ \Big(\frac{c^{p/r}}{\eps\sqrt n}\Big)^{r/\alpha}. 
\]
Since $r \ge 1$ the first term on the right is smaller than a constant times the second
one if $c^{-p/r}\eps\sqrt n$ is small enough.
\end{proof}

\begin{lem}
\label{lemma_balls1}
Let $f \in H^\beta(C)$ for $\beta \leq \alpha+r/2$.
For $\eps > 0$ such that $\eps \to 0$ as $n \to \infty$ and $1/\eps = o(n^{\beta/r})$
and $n$ large enough, 
\begin{equation}
\label{infimum1}
\inf_{h \in \HHH^c: \, \|h-f\|_n\leq \eps} \|h\|^2_{\HHH^c} 
\lle nc^{-(2\alpha + r)/r} \eps^{-\frac{2(\alpha-\beta) + r}{\beta}}. 
\end{equation}
\end{lem}

\begin{proof}
We use an expansion $f = \sum f_{i}\psi_{i}$, with $\psi_i$ the orthonormal eigenfunctions of the Laplacian.
In terms of the coefficients the smoothness assumption reads 
$\sum (1+n^{{2\beta}/r} \lambda_i^\beta) f_{i}^2 \le C^2$. 
Now for $I$ to be determined below, consider $h = \sum_{{i} \le I} f_{i}\psi_{i}$. 
In view of \eqref{eq: lb1}--\eqref{eq: lb2} we have, for $I$ large enough,
\[
\|h-f\|_n^2 = \sum_{{i}> I} f_{i}^2 \le \frac{C^2}{1+ n^{{2\beta}/r}\lambda_I^\beta} \le C^2C_1^{-\beta} I^{-2\beta/r}. 
\]
Setting $I = \text{const} \times \eps^{-r/\beta}$ we get $\|h - f\|_n \le \eps$.
By \eqref{eq: norms}, the RKHS-norm of $h$ satisfies
\begin{align*}
\|h\|^2_{\HHH^c} & = n \sum_{i\le I}((n/c)^{{2}/r}(\lambda_i+n^{-2}))^{\alpha+r/2}f_{i}^2\\
& \lle n {c^{-2p/r}} C^2+ 
 {c^{-2p/r}} C^2 n^{2+2(\alpha-\beta)/r} \lambda_{I}^{p-\beta}. 
\end{align*}
The condition on $\eps$ ensures that for the choice of $I$ made above and $n$ large enough, 
$i_0 \le I \le \kappa n$. Hence, by \eqref{eq: lb1}--\eqref{eq: lb2}, 
$\|h\|_{\HHH^c}^2$ is bounded by a constant times the right-hand side of \eqref{infimum1}.
\end{proof}

\subsubsection{Proof of \eqref{remaining_mass1} and \eqref{entropy1}}

Define $B_n = M_n \HHH_1^{c_n} +\eps_n\BBB_1$, where $\BBB_1$ is the unit ball of $(\RR^n, \|\cdot\|_n)$, 
$\eps_n=n^{-\beta/(r+2\beta)}$ again and $c_n, M_n$ are the sequences to be determined below. 
By Lemma \ref{lem_entropy1} we have 
\begin{equation*}
\log N(2\eps_n, B_n, \|\cdot\|_n) \leq \log N(\eps_n/M_n, \HHH^{c_n}_1, \|\cdot\|_n) \lle
c_n\Big(\frac{M_n}{\eps_n\sqrt n}\Big)^{\frac r{p}},
\end{equation*}
where $p = \alpha + r/2$ again.
For $M_n=M\sqrt{n\eps^2_n}$ and $c^{p/r}_n=N \sqrt n\eps_n^{({\beta-\alpha})/{\beta}}$ this is bounded by a constant times $n\eps^2_n$, which proves \eqref{entropy1}. 

It remains to show that for given $K_2 > 1$, the constants 
$M $ and $N$ can be chosen such that \eqref{remaining_mass1} holds. We have
\[
\PP(f \notin B_n)\leq \int_0^{c_n}\PP(f \notin M_n \HHH_1^{c_n} +\eps_n\BBB_1 \given c) e^{-c}\,dc
+\int_{c_n}^\infty e^{-c}\,dc.
\]
For $c\leq c_n$ we have the inclusion $\HHH_1^c \subseteq \HHH^{c_n}_1$. Hence, by the Borell--Sudakov inequality, 
we have for $c\leq c_n$ that
\begin{align*}
\PP(f \not \in B_n\given c) & \le \PP(f \not \in M_n \HHH_1^{c} +\eps_n\BBB_1\given c)\\
& \le 1-\Phi(\Phi^{-1}(\PP(\|f\|_n \le {\eps_n}\given c)+M_n) )\\
& \le 1-\Phi(\Phi^{-1}(\PP(\|f\|_n \le {\eps_n}\given c_n)+M_n)),
\end{align*}
where $\Phi$ is the cdf of the standard normal distribution.
By Lemma \ref{lemma_small1} the small ball probability in this expression is for 
$c^{p/r}_n=N \sqrt n\eps_n^{({\beta-\alpha})/{\beta}}$ bounded from below by
$\exp(-K\eps_n^{-r/\beta})$ for some $K > 0$. 
Using the bound $\Phi^{-1}(y)\ge -({ (5/2) \log(1/y)})^{1/2}$ for $y \in(0, 1/2)$, it follows that for $c \le c_n$, 
\[
\PP(f \not \in B_n\given c) \le 1 - \Phi(M_n-K' \eps_n^{-r/(2\beta)})
\]
for some $K'>0$. For $M_n$ a large enough multiple of that $\eps_n^{-r/(2\beta)}$ 
this is bounded by $\exp(-K_2\eps_n^{-r/ \beta})=\exp(-K_2 n \eps_n^2)$.

\begin{lem}
\label{lem_entropy1}
For $n$ large enough and $c, \eps > 0$ we have 
\begin{equation}
\label{entropy_1}
\log N(\eps, \HHH_1^{c}, \|\cdot\|_n) \lle c\Big(\frac{1}{\eps\sqrt n}\Big)^{\frac r{\alpha + r/2}}.
\end{equation}
\end{lem}

\begin{proof}
By expanding the RKHS elements in the eigenbasis of the Laplacian
and taking into account the relations \eqref{eq: norms} we see that 
the problem is to bound the entropy $\log N(\eps, A, \|\cdot\|)$,
where
\[
A = \Big\{x \in \RR^n: n \sum_{i=0}^{n-1} ((n/c)^{{2}/r}(\lambda_i+n^{-2})^{\alpha+r/2}x_{i}^2 \le 1\Big\}. 
\]
Using the bounds 
\eqref{eq: lb1}--\eqref{eq: lb2},
it follows that with $p = \alpha + r/2$ and $R = c^{p/r}n^{-(\alpha+r)/r}$ we have the inclusions
\begin{align*}
A & \subset \Big\{x \in \RR^n: \sum_{i=0}^{n-1} \lambda_i^{p}x_{i}^2 \le R^2\Big\}\\
& \subset \Big\{x \in \RR^n: \sum_{i\le i_0} x_{i}^2 \le C_1^{-p}n^{2p}R^2, \quad
\sum_{i > i_0} i^{2p/r}x_i^2\le C_1^{-p}n^{2p/r}R^2\Big\}.
\end{align*}
By using the well-known entropy bounds for balls in $\RR^{i_0}$ and ellipsoids 
in $\ell^2$ we 
deduce from this that for $\eps > 0$, 
\[
\log N(2\eps, A, \|\cdot\|) \lle \log_+\Big(\frac{n^{p}R}{\eps}\Big)
+ \Big(\frac{n^{p/r}R}{\eps}\Big)^{r/p} \lle \Big(\frac{n^{p/r}R}{\eps}\Big)^{r/p}.
\]
The proof is completed by recalling the expressions for $p$ and $R$.
\end{proof}

\subsection{Proof of Theorem \ref{main_theorem2}}
In this case the eigenvalues of $\Sigma_c^{-1}$ are given by 
\[
\mu_i = n^{-1}e^{(n/c)^{2/r}\lambda_i}.
\]

\subsubsection{Proof of \eqref{prior_mass}}

By Lemma 5.3 of \cite{RKHS}, it follows from Lemmas \ref{lemma_small} and \ref{lemma_balls} ahead that 
under the conditions of the theorem and for $\eps = \eps_n = (n/\log^{1+r/2}n)^{-\beta/(r+2\beta)}$ and 
$n\eps^2/\log^{1+r/2} n \le c \le 2n\eps^2/\log^{1+r/2} n$, we have
\[
-\log P(\|f - f_0\|_n \le \eps \given c) \lle c \log^{1+r/2} \frac{c}{\eps^2} + e^{K c^{-2/r}\eps^{-2/\beta}}
\lle n\eps^2.
\]
By conditioning, similar as in the previous case, we find that 
$-\log P(\|f - f_0\|_n \le \eps) \lle n\eps^2$
as well.

\begin{lem}
\label{lemma_small}
If $\eps \to 0$, $c$ is bounded away from $0$ and $c/\eps^2 \to \infty$, then 
\[
-\log P(\|f\|_n \le \eps \given c) \lle c \log^{1+r/2} \frac{c}{\eps^2}.
\]
\end{lem}

\begin{proof}
Again the series representation of the Gaussian law of $f \given c$ gives 
$\PP(\|f\|_n \leq \eps\given c)=\PP(\sum {e^{-(n/c)^{2/r}\lambda_i} Z^2_i }\leq \eps^2)$,
where the $Z_i$ are independent standard normal random variables. 
By the lower bounds \eqref{eq: lb1}--\eqref{eq: lb2}, 
it follows that 
\begin{align*}
& \PP(\|f\|_n \leq 2\eps\given c)\\
& \qquad \ge \PP\Big(\sum_{i \le i_0} e^{-C_1n^{(2-2r)/r}c^{-2/r}}Z^2_i \le \eps^2\Big)
\PP\Big(\sum_{i \ge 1} e^{-C_1c^{-2/r}i^{2/r}}Z^2_i \le \eps^2\Big).
\end{align*}
The first probability in the last line is bounded from below by 
\[
\Big(\PP(|Z_1| < i_0^{-1/2}e^{({1/2})C_1n^{(2-2r)/r}c^{-2/r}}\eps)\Big)^{i_0}.
\]
Since the quantity on the right of the inequality in this probability becomes 
arbitrarily small under de conditions of the lemma, this is further bounded form below
by a constant times $\eps^{i_0}\exp(i_0 (({1/2})C_1n^{(2-2r)/r}c^{-2/r}))$.

For the second probability we use 
Theorem 6.1 of \cite{li_shao}.
This asserts that if
$a_k>0$ and $\sum a_k < \infty$, then as $\eps \to 0$
\begin{equation}\label{eq: ls}
\PP(\sum a_i Z_i^2 \leq \eps^2) \sim \frac1{\sqrt{4\pi \sum (\frac{a_i \gamma_a}{1+2a_i\gamma_a})^2}}
e^{\eps^2\gamma_a-(1/2) \sum \log(1+2a_i\gamma_a)}, 
\end{equation}
where $\gamma_a=\gamma_a(\eps)$ is uniquely determined, for $\eps>0$ small enough, by the equation
\begin{equation}\label{eq: gamma}
\eps^2=\sum\frac{a_i}{1+2a_i\gamma_a}. 
\end{equation}
We apply \eqref{eq: ls} with $a_i = \exp(-C_1(i/c)^{2/r})$.

We first determine bounds for $\gamma_a$. 
Note that in our case the terms in the sum $S$ on the right
of \eqref{eq: gamma} are decreasing in $i$. It follows that we have the bounds 
\[
\int_1^\infty \frac{1}{e^{C_1(x/c)^{2/r}}+2\gamma_a}\,dx \le S \le
\int_0^\infty \frac{1}{e^{C_1(x/c)^{2/r}}+2\gamma_a}\,dx.
\]
A change of variables shows that the integral on the right equals
\[
 \frac{cr}{2C_1^{r/2}} \int_0^\infty \frac{t^{r/2-1}}{e^t + 2\gamma_a}\,dt = 
c \frac{-r \Gamma(r/2)}{4{\gamma_a}C_1^{r/2}} \text{Li}_{r/2}(-2\gamma_a), 
\]
where Li$_s(z)$ denotes the polylogarithm. 
By \cite{Wood}, 
\[
\frac{\text{Li}_{r/2}(-2\gamma_a)}{\log^{r/2}2\gamma_a} \to - \frac{1}{\Gamma(r/2+1)}
\]
as $\gamma_a \to \infty$. Hence for large $\gamma_a$, we have the upper bound
$S \le \text{const} \times {c{\gamma_a}^{-1}\log^{r/2} \gamma_a}$.
It is easily seen that we have a lower bound of the same order, so that 
\[
\eps^2 \asymp \frac{c\log^{r/2} \gamma_a}{\gamma_a}.
\]
Under our condition that $\eps^2/c \to 0$ this holds if and only if 
\[
\gamma_a \asymp \frac{c}{\eps^2} \log^{r/2}\frac{c}{\eps^2}.
\]
%

Next we consider the sums appearing on the right of \eqref{eq: ls}.
To bound $\sum\log(1+2a_i\gamma_a) \le \sum\log(1+2\exp(-C_1(i/c)^{2/r})\gamma_a)$ we 
consider the index $I= c(\log\gamma_a/C_1)^{r/2}$, which is determined such that $ a_I \gamma_a= 1$. 
Note that for $m > 0$, we have $a_{mI}\gamma_a =a_I^{m^{2/r}}\gamma_a = \gamma_a^{1-m^{r/2}}$.
We first split up the sum, writing 
\[
\sum\log(1+2a_i\gamma_a) = \sum_{i < I}\log(1+2a_i\gamma_a)+\sum_{i \ge I}\log(1+2a_i\gamma_a)
\]
The first sum on the right is bounded by a multiple of $I\log\gamma_a$. 
The second one we split into blocks of length $I$. This gives 
\begin{align*}
\sum_{i \ge I}\log(1+2a_i\gamma_a) & \le I \sum_{m \ge 1} \log(1+2\gamma_a^{1-m^{r/2}}) \lle I.
\end{align*}
Hence, we have $\sum\log(1+2a_i\gamma_a) \lle c\log^{1+r/2}\gamma_a$.
For the other sum appearing in \eqref{eq: ls} we have 
\[
\sum \Big(\frac{2a_i \gamma_a}{1+2a_i\gamma_a}\Big)^2\le 
\sum \frac{2a_i \gamma_a}{1+2a_i\gamma_a} = 2\gamma_a \eps^2. 
\]
The proof is completed by combining all the bounds we have found. 
\end{proof}

\begin{lem}
\label{lemma_balls}
Suppose that $f \in H^\beta(C)$ for some $\beta, C > 0$. For $\eps >0$ such that $\eps \to 0$ as $n \to \infty$ and $1/\eps = o(n^{\beta/r})$ and $c > 0$, 
\begin{equation}
\label{infimum2}
\inf_{h \in \HHH^c: \, \|h-f\|_n\le\eps} \|h\|^2_{\HHH^c} \lle
e^{K c^{-2/r}\eps^{-2/\beta}}
\end{equation}
for $n$ large enough, where $K >0 $ is a constant.
\end{lem} 

\begin{proof}
We use an expansion $f = \sum f_{i}\psi_{i}$, with $\psi_i$ the orthonormal eigenfunctions of the Laplacian.
We saw in the proof of Lemma \ref{lemma_balls1} that if we define
$h = \sum_{{i} \le I} f_{i}\psi_{i}$ for $I = \text{const} \times \eps^{-r/\beta}$, then
 $\|h - f\|_n \le \eps$.
By \eqref{eq: norms}, the RKHS-norm of $h$ satisfies in this case
\begin{align*}
\|h\|^2_{\HHH^c} & = \sum_{i \le I} e^{(n/c)^{2/r}\lambda_i}f_i^2 \le C^2e^{(n/c)^{2/r}\lambda_I}. 
\end{align*}
The condition on $\eps$ ensures that for the choice of $I$ made above and $n$ large enough, 
$i_0 \le I \le \kappa n$. Hence, by \eqref{eq: lb1}--\eqref{eq: lb2}, 
$\|h\|_{\HHH^c}^2$ is bounded by a constant times 
the right-hand side of \eqref{infimum2}. 
\end{proof}

\subsubsection{Proof of \eqref{remaining_mass}--\eqref{entropy}}

Define $B_n:= M_n \HHH_1^{c_n} +\eps_n\BBB_1$, where $\eps_n$
is as above and $M_n$ and $c_n$ are determined below.

For \eqref{remaining_mass} we first note again that 
\[
\PP(f \notin B_n)\leq \int_0^{c_n}\PP(f \notin M_n \HHH_1^{c_n} +\eps_n\BBB_1 \given c) e^{-c}\,dc
+\int_{c_n}^\infty e^{-c}\,dc.
\]
Exactly as in the proof of \eqref{remaining_mass1}, the Borell--Sudakov inequality
implies that for $c\leq c_n$, 
\begin{align*}
\PP(f \not \in B_n\given c) \le 1-\Phi(\Phi^{-1}(\PP(\|f\|_n \le {\eps_n}\given c_n)+M_n)).
\end{align*}
By Lemma \ref{lemma_small} the small ball probability on the right is lower bounded by 
\[
\exp\Big(-Kc_n \log^{1+r/2} \frac{c_n}{\eps_n^2}\Big).
\]
It follows that for $c \le c_n$, 
\[
\PP(f \not \in B_n\given c) \le 1 - \Phi\Big(M_n-K' \sqrt{c_n \log^{1+r/2} \frac{c_n}{\eps_n^2}}\Big)
\]
for some $K'>0$. For a given $K_2 > 0$, choosing $M_n$ a large multiple of 
$(c_n \log^{1+r/2} ({c_n}/{\eps_n^2}))^{1/2}$
we find that, for large $n$, 
\[
\PP(f \not \in B_n) \le e^{-K''{c_n \log^{1+r/2} \frac{c_n}{\eps_n^2}} } + e^{-c_n} \le 2e^{-c_n}.
\]
If $K_2 > 0$ is a given constant, then for $c_n$ a large enough multiple of $n\eps_n^2$, this is bounded by 
$\exp(-K_2n\eps_n^2)$.

For these choices of $M_n$ and $c_n$, Lemma \ref{lem_entropy} implies that the entropy satisfies, for any $\tilde \eps_n \ge \eps_n$, 
\[
\log N(2\tilde\eps_n, B_n, \|\cdot\|_n) \le \log N(2\eps_n, B_n, \|\cdot\|_n) 
 \lle
c_n \left(\log \frac{M_n} {\eps_n}\right)^{1+r/2}.
\]
This proves that \eqref{entropy} holds for 
$\tilde \eps_n = \eps_n\log^{1/2+r/4}n$.

\begin{lem}
\label{lem_entropy}
Let $\eps, c > 0$ be such that $c\log^{r/2}(1/\eps) \to \infty$ as $n \to \infty$.
Then for $n$ large enough, 
\[
\log N(\eps, \HHH_1^c, \|\cdot\|_n) \lle c\log^{1+r/2}\Big(\frac1\eps\Big). 
\]
\end{lem}

\begin{proof}
We need to bound the metric entropy of the set 
\[
A = \{x \in \RR^n: \sum\limits_{i=0}^{n-1}e^{(n/c)^{2/r}\lambda_i}x_{i}^2 \le 1\}, 
\]
relative to the Euclidean norm $\|\cdot\|$. Set 
$I = (2/{C_1})^{r/2} c\log^{r/2}(1/\eps)$.
Under the assumption of the lemma this is larger than $i_0$, hence by \eqref{eq: lb1}--\eqref{eq: lb2}
we have $\exp(-(n/c)^{2/r}\lambda_I) \le \eps^2$. 
It follows that if for $x \in A$ we define the projection $x^I$ by 
$x^I = (x_1, \ldots, x_I, 0, 0, \ldots)$, then 
\[
\|x - x^I\|^2 = \sum_{i > I} x_i^2 \le e^{-(n/c)^{2/r}\lambda_I} \sum_{i> I}e^{(n/c)^{2/r}\lambda_i}x_{i}^2 \le \eps^2.
\]
Moreover, we have $\|x^I\| \le 1$.
By the triangle inequality, it follows that if the points $x_1, \ldots, x_N$ form an $\eps$-net for 
the unit ball in $\RR^I$, 
then the points $\bar x_1, \ldots, \bar x_N$ in $\RR^n$ obtained by appending zeros to the $x_j$ form
a $2\eps$-net for $A$. Hence, $N(2\eps, A, \|\cdot\|) \lle \eps^{-I}$. 
The proof is completed by recalling the expression for $I$.
\end{proof}

\section{Function estimation on graphs}
\label{sec: stat}

In this section we translate the general Theorems \ref{main_theorem} and \ref{main_theorem2}
into results about posterior contraction in nonparametric regression and binary 
classification problems on graphs. Since the arguments needed for this translation are 
very similar to those in earlier papers, we omit full proofs and just give pointers to the literature.

\subsection{Nonparametric regression}

As before we let $G$ be a connected simple undirected graph with vertices $1, 2, \ldots, n$. 
In the regression case we assume that we have observations $Y_1, \ldots, Y_n$ at the vertices
of the graph, satisfying
\begin{equation}\label{eq: reg}
Y_i = f_0(i) + \eps_i, 
\end{equation}
where $f_0$ is the function on $G$ that we want to make inference about and $\eps_i$ are independent 
$N(0,\sigma^2)$-distributed error variables, for some $\sigma > 0$.
We assume that the underlying graph satisfies the geometry assumption with some parameter $r \ge 1$. 
As prior on the regression function $f$ we then employ one of the two priors described by 
\eqref{eq: c1}--\eqref{eq: f1} or \eqref{eq: c2}--\eqref{eq: f2}. If $\sigma$ is unknown, we 
assume it belongs to a compact interval $[a,b] \subset (0,\infty)$ and endow it with a prior 
with a positive, continuous density on $[a,b]$, independent of the prior on $f$. 

For a given prior $\Pi$, the corresponding posterior distribution on $f$ is denoted by 
$\Pi(\cdot \given Y_1, \ldots, Y_n)$. For a sequence of positive numbers $\eps_n \to 0$
we say that the posterior contracts around $f_0$ at the rate $\eps_n$ if for all large enough $M> 0$, 
\[
\Pi(f: \|f-f_0\|_n \ge M\eps_n \given Y_1, \ldots, Y_n) \overset{P_{f_0}}{\to} 0
\]
as $n \to \infty$. Here the convergence is in probability under the law $P_{f_0}$ 
corresponding to the data generating model \eqref{eq: reg}.

The usual arguments allow us to derive the following statements from Theorems 
\ref{main_theorem} and \ref{main_theorem2}. See, for instance, \cite{vdVvZ}
or \cite{rene} for details.

\begin{thm}[Nonparametric regression]
\label{regression_theorem}
Suppose the geometry assumption holds for $r \ge 1$. 
Assume that $f_0 \in H^\beta(C)$ for $\beta, C > 0$. 
\begin{enumerate}[(i)]
\item (Power of the Laplacian.) If the prior on $f$ is given by \eqref{eq: c1}--\eqref{eq: f1} for $\alpha > 0$ 
and $\beta \le \alpha + r/2$, then the posterior contracts around $f_0$ at the rate 
$n^{-\beta/(r+2\beta)}$.

\item (Exponential of the Laplacian.) If the prior on $f$ is given by \eqref{eq: c2}--\eqref{eq: f2}, then the posterior contracts around $f_0$ at the rate $n^{-\beta/(r+2\beta)}\log^\kappa n$ for some $\kappa > 0$.
\end{enumerate}
\end{thm}

Observe that since the priors do not use knowledge of the regularity $\beta$ 
of the regression function, we obtain rate-adaptive results. For the power prior the 
range of regularities that we can adapt to is bounded by $\alpha + r/2$, where $\alpha$
is the hyper parameter describing the ``baseline regularity'' of the prior. In the 
case of the exponential prior the range is unbounded. This comes at the modest cost 
of having an additional logarithmic factor in the rate. 

In \cite{minimax}  minimax lower
bounds are presented which complement the rate results of the present paper.
These show that the rates obtained are sharp 
in the present setting (up to a logarithmic factor in the exponential case). 
For the regular grid case this is basically also clear from existing lower bound results,  
since our setup includes the regular grids (Example \ref{ex: grids}) and since 
our smoothness condition corresponds to ordinary Sobolev regularity in those cases (Example \ref{ex: pathgraph}).

\subsection{Nonparametric classification}

We can derive the analogous results in the classification problem in which we assume that the data $Y_1, \ldots, Y_n$ are 
independent $\{0, 1\}$-valued variables, observed at the vertices of the graph. 
In this case the goal is to estimate the binary regression function $p_0$, or ``soft label function'' on the 
graph, given by 
\[
p_0(i) = \PP_0(Y_i = 1). 
\]
We consider priors on $p$ constructed by first defining a prior on a real-valued function 
$f$ by \eqref{eq: c1}--\eqref{eq: f1} or \eqref{eq: c2}--\eqref{eq: f2} and then setting 
$p = \Psi(f)$, where $\Psi: \RR \to (0, 1)$ is a suitably chosen link function.
We will assume that $\Psi$ is a strictly increasing, differentiable function onto $(0,1)$ such that 
${\Psi'}/({\Psi(1-\Psi)})$ is uniformly bounded. Note that for instance the 
sigmoid, or logistic link $\Psi(f) = 1/(1+\exp(-f))$ satisfies this condition. 
Under our conditions the inverse $\Psi^{-1}: (0,1)\to \RR$ is well defined. In this classification setting
the regularity condition will be formulated in terms of $\Psi^{-1}(p_0)$. This is natural, 
since the prior is defined in terms of $\Psi^{-1}(p)$ as well. 
Also in this case we denote the posterior corresponding to a prior $\Pi$ by 
$\Pi(\cdot \given Y_1, \ldots, Y_n)$ and we say that 
 the posterior contracts around $p_0$ at the rate $\eps_n$ if for all large enough $M> 0$, 
\[
\Pi(p: \|p-p_0\|_n \ge M\eps_n \given Y_1, \ldots, Y_n) \overset{P_{0}}{\to} 0
\]
as $n \to \infty$. 

To derive the following result from Theorems \ref{main_theorem} and \ref{main_theorem2} 
we can argue along the lines of the proof of Theorem 3.2 of \cite{vdVvZ}. 
Some adaptations are required, since in the present case we have fixed design points. However, the 
necessary modifications are straightforward and therefore omitted. 

\begin{thm}[Classification]
\label{label_theorem}
Suppose the geometry assumption holds for $r \ge 1$. 
Let $\Psi: \RR \to (0,1)$ be onto, strictly increasing, differentiable and suppose that 
${\Psi'}/({\Psi(1-\Psi)})$ is uniformly bounded.
Assume that $\Psi^{-1}(p_0) \in H^\beta(C)$ for $\beta, C > 0$. 
\begin{enumerate}[(i)]
\item (Power of the Laplacian.) If the prior on $p$ is given by the law of $\Psi(f)$, for $f$ given by 
\eqref{eq: c1}--\eqref{eq: f1} for $\alpha > 0$ 
and $\beta \le \alpha + r/2$, then the posterior contracts around $p_0$ at the rate 
$n^{-\beta/(r+2\beta)}$.

\item (Exponential of the Laplacian.) If the prior on $p$ is given by the law of $\Psi(f)$, for $f$ given by 
\eqref{eq: c2}--\eqref{eq: f2}, then the posterior contracts around $f_0$ at the rate $n^{-\beta/(r+2\beta)}\log^\kappa n$ for some $\kappa > 0$.
\end{enumerate}
\end{thm}

\section{Concluding remarks}
\label{sec: conc}

We have introduced a framework for studying the performance 
of methods for nonparametric function estimation on large graphs. 
We have proposed assumptions on the geometry of the underlying graph 
and the regularity of the function formulated in terms of the 
Laplacian of the graph. Moreover, we have exhibited nonparametric Bayes
methods that achieve good convergence rates and that adapt to the 
unknown regularity of the function of interest.

We have purposely focused on the building up a new framework and 
deriving a few representative results within that framework and have 
not yet attempted to explore every possible extension. 
As a result, extensions and generalizations are possible in a variety of directions.

First of all, it is of interest to study other procedures than 
just the Bayesian methods with priors \eqref{eq: c1}--\eqref{eq: f1} or \eqref{eq: c2}--\eqref{eq: f2}.
For instance, empirical Bayes procedures for choosing the hyperparameter $c$ 
might computationally be more favorable than hierarchical Bayes. Studying the performance of such procedures 
is possible within the framework of \cite{rousseau2015asymptotic}.
In turn, having results for empirical Bayes will 
allow us to extend the range of priors on $c$ for which we can prove that the hierarchical 
procedures give good results. 

Secondly, results on uncertainty quantification 
would be valuable. Bayes procedures provide a natural method for quantifying uncertainty 
through the spread of the posterior distribution. However, it has become clear
that in general the question of whether or not Bayesian credible sets can be interpreted as (frequentist) confidence
sets is a delicate matter in nonparametric settings (e.g.\ \cite{conf}).
It would be desirable to have more insight in this issue in the graph setting.

On the level of the geometry assumption, several extensions might be of interest.
For instance, instead of a single parameter $r$ governing the ``dimension''
of the graph it might be interesting to consider frameworks allowing graphs
which are less homogenous. When estimating a function on some sub-region of 
a graph, one would expect that the rates should only depend on the local geometry 
of the graph in that region. It would be of interest to make such statements
mathematically precise and to exhibit procedures with good local properties. 
More generally, recent numerical work has shown that Bayesian Laplacian regularisation 
can work quite  well in practice on graphs that do not satisfy our geometry assumption, 
see \cite{Jarno}. To understand this theoretically our current mathematical results are too
limited.

A final possible generalization that we want to mention is to the setting of 
 weighted graphs. This is of interest, since in many applications it is 
 natural to work with weighted graphs to quantify the similarity between vertices. 
 We expect that with additional work our results can be extended to that setting.

\bigskip

\newpage

\bibliographystyle{harry}
\bibliography{pathgraph}

\end{document}